\documentclass[11pt]{amsart}

\usepackage[utf8]{inputenc}
\usepackage[T1]{fontenc}
\usepackage[english]{babel}

\usepackage{amsfonts}
\usepackage{amsmath}
\usepackage[alphabetic]{amsrefs}
\usepackage{amssymb}
\usepackage{amstext}
\usepackage{amsthm}
\usepackage{fullpage}

\usepackage{hyperref}
\usepackage{mathrsfs}
\usepackage{mathtools}
\usepackage{multicol}

\usepackage{tikz}
\usetikzlibrary{backgrounds,shapes}

\theoremstyle{plain}

\theoremstyle{definition}
\newtheorem{theorem}{Theorem}[section]

\newtheorem{corollary}[theorem]{Corollary}
\newtheorem{definition}[theorem]{Definition}
\newtheorem{lemma}[theorem]{Lemma}
\newtheorem{proposition}[theorem]{Proposition}

\theoremstyle{remark}
\newtheorem{example}[theorem]{Example}
\newtheorem{remark}[theorem]{Remark}

\newcommand{\area}{\mathsf{area}}
\newcommand{\dinv}{\mathsf{dinv}}
\newcommand{\inv}{\mathsf{inv}}
\newcommand{\bounce}{\mathsf{bounce}}

\numberwithin{equation}{section}

\title{$e$-positivity of vertical strip LLT polynomials}

\author[M. D'Adderio]{Michele D'Adderio}
\address{Universit\'e Libre de Bruxelles (ULB), D\'epartement de Math\'ematique, Boulevard du Triomphe, B-1050 Bruxelles, Belgium}\email{mdadderi@ulb.ac.be}

\begin{document}
	
\dedicatory{Dedicated to Adriano Garsia on the occasion of his 90th birthday}

\maketitle

\begin{abstract}
In this article we prove the $e$-positivity of $G_{\mathbf{\nu}}[X;q+1]$ when $G_{\mathbf{\nu}}[X;q]$ is a vertical strip LLT polynomial. This property has been conjectured in \cite{Alexandersson_Panova_cycles} and \cite{Garsia_Haglund_Qiu_Romero}, and it implies several $e$-positivities conjectured in those references and in \cite{Bergeron_Open_Questions}.
	
We make use of a result of Carlsson and Mellit \cite{Carlsson-Mellit-ShuffleConj-2015} that shows that a vertical strip LLT polynomial can be obtained by applying certain compositions of operators of the Dyck path algebra to the constant $1$. Our proof gives in fact an algorithm to expand these symmetric functions in the elementary basis, and it shows, as a byproduct, that these compositions of operators are actually multiplication operators.
\end{abstract}

\tableofcontents

\section*{Introduction}

In \cite{Bergeron_Open_Questions} Bergeron conjectured that several symmetric functions $G[X;q]$ arising from the theory of Macdonald polynomials have the property that $G[X;q+1]$ is \emph{$e$-positive}, i.e.\ the coefficients of the expansion of $G[X;q+1]$ in the elementary symmetric function basis are in $\mathbb{N}[q,t]$. Later similar conjectures have been made in \cite{Garsia_Haglund_Qiu_Romero}. Many of these conjectures would follow easily if that same property were true for vertical strip LLT polynomials. This last property has been conjectured both in \cite{Alexandersson_Panova_cycles} and \cite{Garsia_Haglund_Qiu_Romero}, and it was observed in \cite{Alexandersson_Panova_cycles} that it has an interesting parallel with a famous $e$-positivity conjecture about the \emph{chromatic symmetric functions} introduced by Shareshian and Wachs in \cite{Shareshian_Wachs_Original}. Notice also that in \cite{Alexandersson_Panova_cycles}, \cite{Garsia_Haglund_Qiu_Romero} and \cite{Alexandersson_lollipop} some special cases have been proved. 

The LLT polynomials were introduced by Lascoux, Leclerc and Thibon in \cite{LLT_Original}, and they can be seen as a $q$-deformation of products of skew Schur functions. These symmetric functions $G_{\mathbf{\nu}}[X;q]$ are associated to a tuple $\mathbf{\nu}$ of skew Young diagrams. It turns out that many of the symmetric functions arising in the study of modified Macdonald polynomials exhibit a natural expansion in LLT polynomials with coefficients in $\mathbb{N}[q,t]$. In fact, a lot of these expansions only involve vertical strip LLT polynomials.

When all the skew shapes in $\mathbf{\nu}$ are (continuous) vertical strips of cells, we say that $G_{\mathbf{\nu}}[X;q]$ is a \emph{vertical strip LLT polynomial}. The main result of this article is the proof of the $e$-positivity of $G_{\mathbf{\nu}}[X;q+1]$ when $G_{\mathbf{\nu}}[X;q]$ is a vertical strip LLT polynomial. In fact our proof gives an algorithm to compute the expansion of vertical strip LLT polynomials in the elementary symmetric function basis.

The main ingredient of our proof is a result of Carlsson and Mellit in \cite{Carlsson-Mellit-ShuffleConj-2015}, which states that these symmetric functions can be computed by applying a certain composition of operators coming from their \emph{Dyck path algebra} to the constant function $1$. As a corollary of our proof, we get that these compositions of operators are actually multiplication operators. We notice here that this fact is probably related to the same property conjectured by Bergeron in \cite{Bergeron_Open_Questions} for the specialization at $t=1$ of related symmetric functions, though we do not discuss this further in this article.

As a corollary of our main result, we get many of the $e$-positivity conjectured in \cite{Bergeron_Open_Questions}, \cite{Alexandersson_Panova_cycles} and \cite{Garsia_Haglund_Qiu_Romero}.
But we want to stress here that not all the $e$-positivities observed in \cite{Bergeron_Open_Questions}, \cite{Alexandersson_Panova_cycles} and \cite{Garsia_Haglund_Qiu_Romero} follow from our results. Moreover, in
\cite{Alexandersson_lollipop} and \cite{Garsia_Haglund_Qiu_Romero} the authors conjecture explicit formulas for the expansion in the elementary symmetric functions basis of $G_{\mathbf{\nu}}[X;q+1]$ when $G_{\mathbf{\nu}}[X;q]$ is a vertical strip LLT polynomial. We do not say anything about these formulas, though it is conceivable that a deeper analysis of our algorithmic proof could yield new insights into these problems. We leave these considerations for future investigations.

\medskip

The paper is organized in the following way. In Section~1 we introduce some notation from symmetric function theory, while in Section~2 we define vertical strip LLT polynomials and we associate to them certain Schr\"{o}der paths. In Section~3 we introduce the operators of the Dyck path algebra, and we collect some identities that we are going to use in the sequel. In Section~4 we define the path operators corresponding to the Schr\"{o}der paths defined in Section~2, and we state the theorem of Carlsson and Mellit. In Section~5 we prove our main results. Finally in Section~6 we discuss some consequences.

\section{Symmetric function notations}

In this section we limit ourselves to introduce the necessary notation to state our results. We refer to \cite{Stanley-Book-1999} and \cite{Haglund-Book-2008} for the basic symmetric function theory tools that we use freely in the sequel.

We denote by $\Lambda$ the algebra of symmetric functions in the variables $x_1,x_2,\dots$ with coefficients in the field $\mathbb{Q}(q,t)$. We denote by $e_k$ the \emph{elementary symmetric function} of degree $k$, and given a partition $\mu=(\mu_1,\mu_2,\dots)$ we set $e_\mu=e_{\mu_1}e_{\mu_2}\cdots$. We use similar notations for the \emph{power symmetric functions} $p_k$'s.

It is well-known that both $\{p_\mu\}_\mu$ and $\{e_\mu\}_\mu$ are bases of $\Lambda$.

\begin{definition}
	We say that a symmetric function $f\in \Lambda$ is \emph{$e$-positive} if the coefficients of the expansion of $f$ in the basis $\{e_\mu\}_\mu$ are all in $\mathbb{N}[q,t]$.
\end{definition}

We will make also use of the \emph{plethystic notation}. With this notation we will be able to add and subtract alphabets, which will be represented as sums of monomials $X = x_1 + x_2 + x_3+\cdots $. Then, given a symmetric function $f$, we denote by $f[X]$ the expansion of $f$ in the basis $\{p_\mu\}_\mu$ with $p_k$ replaced by $x_{1}^{k}+x_{2}^{k}+x_{3}^{k}+\cdots$, for all $k$. More generally, given any expression $Q(z_1,z_2,\dots)$, we define the plethystic substitution $f[Q(z_1,z_2,\dots)]$ to be the expansion of $f$ in the basis $\{p_\mu\}_\mu$ with $p_k$ replaced by $Q(z_1^k,z_2^k,\dots)$.

\section{Vertical strip LLT polynomials} \label{sec:LLT}

In this section we define the main characters of this article, i.e.\ the vertical strip LLT polynomials.

We use standard definitions and French notations for Young (Ferrers) diagrams and partitions: e.g.\ see \cite{Haglund-Book-2008}.

\subsection{Definition of LLT polynomials}

We identify a partition $\lambda=(\lambda_1,\lambda_2,\dots,\lambda_k)$ with its \emph{Ferrers diagram}, i.e.\ the set of unit squares in $\mathbb{R}^2$, called \emph{cells}, with centers $\{(i,j)\in \mathbb{Z}^2\mid 1\leq i\leq \lambda_j, 1\leq j\leq k  \}$ (we identify these with the corresponding cells). We will consider also \emph{skew diagrams} $\lambda/\mu$ with $\mu\subseteq \lambda$, and \emph{semi-standard Young tableaux} of \emph{shape} $\lambda/\mu$, i.e.\ fillings $T:\lambda/\mu\to \mathbb{Z}_{>0}$ such that $T((i,j))<T((i+1,j))$ whenever $(i,j),(i+1,j)\in \lambda/\mu$, and $T((i,j))\leq T((i,j+1))$ whenever $(i,j),(i,j+1)\in \lambda/\mu$. We denote by $SSYT(\lambda/\mu)$ the set of skew semi-standard Young tableaux of shape $\lambda/\mu$. Given $T\in SSYT(\lambda/\mu)$ we set $\mathbf{x}^T:=\prod_{c\in \lambda/\mu}x_{T(c)}$.

We define the LLT polynomials as in \cite{HHL_JAMS}.

\begin{definition}
	Given a $k$-tuple $\mathbf{\nu}=(\mathbf{\nu}^1,\dots,\mathbf{\nu}^k)$ of skew Young diagrams, we let $SSYT(\mathbf{\nu})=SSYT(\mathbf{\nu}^1)\times \cdots  \times SSYT(\mathbf{\nu}^k)$. Given $T=(T^1,\dots,T^k)\in SSYT(\mathbf{\nu})$, let $\mathbf{x}^T$ denote the product $\mathbf{x}^{T^1}\cdots \mathbf{x}^{T^k}$. We say that $T^i(u)<T^j(v)$ form an \emph{inversion} if either
	\begin{itemize}
		\item $i<j$ and $c(u)=c(v)$, or
		\item $i>j$ and $c(u)=c(v)+1$,
	\end{itemize} 
	where $c(u)$ is the \emph{content} of $u$: if $u$ has (row, column) coordinates $(i,j)$, then $c(u):=i-j$. 
	
	We call $\inv(T)$ the total number of inversions occurring in $T$.
\end{definition}
\begin{remark}
	Notice that our convention differs from the one in \cite{HHL_JAMS}: to go from one convention to the other it is enough to reverse the order of the components of the tuple $\mathbf{\nu}$.
\end{remark}
We define the \emph{LLT polynomial} 
\begin{equation}
G_{\mathbf{\nu}}[X;q]:=\sum_{T\in SSYT(\mathbf{\nu})}q^{\inv(T)}\mathbf{x}^T.
\end{equation}

It turns out that the $G_{\mathbf{\nu}}[X;q]$'s are symmetric functions, i.e.\ elements of $\Lambda$, though this is far from obvious from their definition: see \cite{HHL_JAMS} for an elementary proof of this fact.

A \emph{vertical strip} (or \emph{column}) \emph{LLT polynomial} is a $G_{\mathbf{\nu}}[X;q]$ where each skew diagram of $\mathbf{\nu}$ consists of a (continuous) \emph{vertical strip} of cells in the same column.

To visualize a tuple of skew Young diagrams, it is custom to arrange them in the first quadrant, in such a way that cells with the same content appear on the same diagonal in a non-overlapping fashion: we will call this a \emph{LLT diagram} associated to the tuple of skew Young diagrams, or skew Young tableaux.

As an example, the LLT diagram associated to a tuple of vertical strips is shown on the left of Figure~\ref{fig:LLT}, with letters inside the cells. Notice that, if this were the LLT diagram of a tuple of skew Young tableaux, then for example the pair $b,c$ would give an inversion if $b<c$, the pair $c,d$ would give an inversion if $d>c$, we must have $b<d$, and there is no way in which $d$ and $h$ can create an inversion.

\begin{figure*}[!ht]
	\begin{minipage}{.5\textwidth}
		\centering
		\begin{tikzpicture}[scale=.5]			
		\draw[gray!90, thin](0,0) grid (1,2);
		
		\draw[gray!90, thin](2,0) grid (3,2);
		
		\draw[gray!90, thin](4,3) grid (5,4);
		
		\draw[gray!90, thin](5,6) grid (6,7);
		
		\draw[gray!90, thin](6,3) grid (7,5);
		
		\draw[gray!90, thin](7,6) grid (8,8);

		\draw[red!50, thin](0,1) -- (7,8);
		\draw[red!50, thin](0,0) -- (8,8);
		\draw[red!50, thin](1,0) -- (8,7);
		\draw[red!50, thin](2,0) -- (8,6);
		\draw[red!50, thin](3,0) -- (8,5);
		
		\draw
		(.5,.5) node {g}
		(.5,1.5) node {i}
		(4.5,3.5) node {e}
		(2.5,1.5)  node {d}
		(2.5,.5) node {b}
		(5.5,6.5) node {j}
		(6.5,3.5) node {a}
		(6.5,4.5) node {c}
		(7.5,6.5) node {f}
		(7.5,7.5) node {h};
		
		\draw (10.5,3.5) node {$\Rightarrow$};		
		\end{tikzpicture}
		
	\end{minipage}%
	\hspace{-3cm}
	\begin{minipage}{.5 \textwidth}
		\centering
		
		\hspace{2cm}
		\begin{tikzpicture}[scale=0.5]
		\filldraw[yellow, opacity=0.3] (0,0) -- (0,1) -- (1,1) -- (1,2) -- (2,2) -- (2,3) -- (3,3) -- (3,4) -- (4,4) --(4,5) -- (5,5)  --(5,6) -- (6,6) -- (6,7) -- (7,7)  -- (7,8) -- (8,8) -- (8,9) -- (9,9) -- (9,10) -- (10,10) -- (10,9) -- (9,9) -- (9,8) -- (8,8) -- (8,7) -- (7,7) -- (7,6) -- (6,6) --  (6,5) -- (5,5) -- (5,4) -- (4,4) -- (4,3) -- (3,3) -- (3,2) -- (2,2) -- (2,1) -- (1,1) -- (1,0)  ;
		
		\draw[gray!60, thin] (0,0) grid (10,10);
		
		\node[blue!60] at (8.5,9.5) {$\bullet$};
		\node[blue!60] at (7.5,9.5) {$\bullet$};
		\node[blue!60] at (7.5,8.5) {$\bullet$};
		\node[blue!60] at (4.5,6.5) {$\bullet$};
		\node[blue!60] at (4.5,5.5) {$\bullet$};
		\node[blue!60] at (3.5,5.5) {$\bullet$};
		\node[blue!60] at (3.5,4.5) {$\bullet$};
		\node[blue!60] at (2.5,4.5) {$\bullet$};
		\node[blue!60] at (2.5,3.5) {$\bullet$};
		\node[blue!60] at (1.5,2.5) {$\bullet$};
		\node[blue!60] at (0.5,1.5) {$\bullet$};
		\node[blue!60] at (5.5,6.5) {$\bullet$};		
		\node[blue!60] at (6.5,7.5) {$\bullet$};		
		\node[blue!60] at (3.5,6.5) {$\bullet$};
		
		\draw
		(0.5,0.5) node {$a$}
		(1.5,1.5) node {$b$}
		(2.5,2.5) node {$c$}
		(3.5,3.5) node {$d$}
		(4.5,4.5) node {$e$}
		(5.5,5.5) node {$f$}
		(6.5,6.5) node {$g$}
		(7.5,7.5) node {$h$}
		(8.5,8.5) node {$i$}
		(9.5,9.5) node {$j$};
		
		\draw[blue!60, line width = 1pt] (0,2)--(1,3) (0,3)--(1,2);
		\draw[blue!60, line width = 1pt] (1,3)--(2,4) (1,4)--(2,3);
		\draw[blue!60, line width = 1pt] (5,7)--(6,8) (5,8)--(6,7);
		\draw[blue!60, line width = 1pt] (6,8)--(7,9) (6,9)--(7,8);

		\draw[red!60, line width = 1.6pt] (0,0)|-(1,2)|-(2,3)|-(3,5)|-(6,7)|-(7,8)|-(10,10);
		\end{tikzpicture}
	\end{minipage}
	\caption{}
	\label{fig:LLT}
\end{figure*}

\subsection{Schr\"{o}der paths associated to vertical strip LLT diagrams} \label{Sec:Schroeder_LLT}

In order to state (later in this article) a result of Carlsson and Mellit, we describe now a procedure to associate a \emph{Schr\"{o}der path} to a \emph{vertical strip LLT diagram} (i.e.\ the LLT diagram of a tuple of vertical strips). 

We borrow the notation for the pictures from \cite{Garsia_Haglund_Qiu_Romero}, in which more details can be found.

We start by labelling with $a,b,c,d,\dots$ the cells of the vertical strip LLT diagram in \emph{reading order}, i.e.\ from left to right along the diagonals with constant content, starting from the lowest diagonal and moving upward: see Figure~\ref{fig:LLT} on the left for an example. Then we place the letters $a,b,c,d,\dots$ in this order along the diagonal of a square grid: see Figure~\ref{fig:LLT} on the right for an example.

Then for every pair of letters $(p,q)$ with $p$ to the left of $q$ we draw a blue dot in the cell at the intersection of the column containing the letter $p$ with the row containing $q$ if and only if $p$ and $q$ can potentially create an inversion in the vertical strip diagram. It is easy to see that these blue dots will always determine a \emph{Dyck path}, like the red one in Figure~\ref{fig:LLT}. Moreover, for every pair of letters $(p,q)$ with $p$ to the left of $q$ we draw a blue cross if and only if $q$ is in the cell right above the cell containing $p$ in the LLT diagram (therefore we must have $p<q$). It is easy to see that these blue crosses always end up decorating a \emph{valley} that is not on the main diagonal $x=y$ of the aforementioned Dyck path, i.e.\ in the cells immediately east and immediately south of the cell containing the blue cross there are always blue dots: see Figure~\ref{fig:LLT} on the right for an example.

Now replacing the valleys with the blue crosses (together with the two adjacent steps of the Dyck path) by diagonal edges we get a \emph{Schr\"{o}der path} that does not have diagonal steps along the diagonal $x=y$: see Figure~\ref{fig:LLT2} for an example. 

\begin{figure*}[!ht]
		\centering
		
		\begin{tikzpicture}[scale=0.5]
		\filldraw[yellow, opacity=0.3] (0,0) -- (0,1) -- (1,1) -- (1,2) -- (2,2) -- (2,3) -- (3,3) -- (3,4) -- (4,4) --(4,5) -- (5,5)  --(5,6) -- (6,6) -- (6,7) -- (7,7)  -- (7,8) -- (8,8) -- (8,9) -- (9,9) -- (9,10) -- (10,10) -- (10,9) -- (9,9) -- (9,8) -- (8,8) -- (8,7) -- (7,7) -- (7,6) -- (6,6) --  (6,5) -- (5,5) -- (5,4) -- (4,4) -- (4,3) -- (3,3) -- (3,2) -- (2,2) -- (2,1) -- (1,1) -- (1,0)  ;
		
		\draw[gray!60, thin] (0,0) grid (10,10);
			
		\draw[red!60, line width = 1.6pt] (0,0)--(0,2)--(2,4)--(2,5)--(3,5)--(3,7)--(5,7)--(7,9)--(7,10)--(10,10);
		\end{tikzpicture}
	\caption{This is the Schr\"{o}der path corresponding to the diagram on the right in Figure~\ref{fig:LLT}}
	\label{fig:LLT2}
\end{figure*}
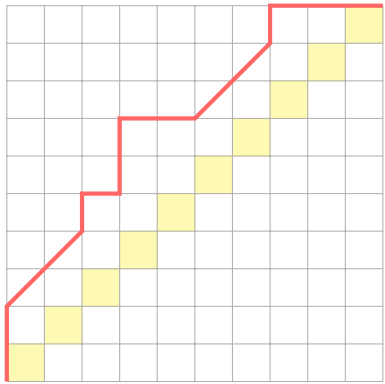

It is not hard to see that any such Schr\"{o}der path comes from a vertical strip LLT diagram via this construction. For more details on all this, see \cite{Garsia_Haglund_Qiu_Romero} or \cite{Alexandersson_Panova_cycles}. 

We call $\mathcal{P}$ the set of these paths. More precisely, we denote by $\mathcal{P}$ the family of \emph{paths} $P$ going from $(0,0)$ to $(n,n)$ for some $n\in \mathbb{N}$ which consist of vertical north steps, obtained by adding $(0,1)$, horizontal east steps, obtained by adding $(1,0)$, and diagonal northeast steps, obtained by adding $(1,1)$, with the property that $P$ always stays weakly above the \emph{base diagonal} $y=x$, and no diagonal step can take place along that line.

Given a tuple $\mathbf{\nu}$ of vertical strips, we call $P_{\mathbf{\nu}}$ the element of $\mathcal{P}$ associated to $\mathbf{\nu}$ via the construction that we explained in this subsection.
\begin{remark}
As we will only use this construction to state Theorem~\ref{thm:CM_LLT}, we omit the discussion of the classical combinatorial bijection (sending $(\dinv,\area)$ into $(\area',\bounce)$) lying behind it: see \cite{Carlsson-Mellit-ShuffleConj-2015} or \cite{Garsia_Haglund_Qiu_Romero} for more on this.
\end{remark}

\section{Dyck path algebra operators}

In this section we introduce the operators of the Dyck path algebra of Carlsson and Mellit, and we collect some basic properties that we are going to use in the sequel.

\medskip

We start by recalling some definitions from \cite[Section~4]{Carlsson-Mellit-ShuffleConj-2015}.

Given a polynomial $P$ depending on variables $u,v$, define the operator $\Upsilon_{uv}$ as

\begin{align*}
	(\Upsilon_{uv} P)(u,v) & \coloneqq \frac{(q-1)vP(u,v) + (v-qu)P(v,u)}{v-u}, 
\end{align*}

In \cite{Carlsson-Mellit-ShuffleConj-2015} these operators are called $\Delta_{uv}$, but we changed the notation in order to avoid confusion with the Delta operators $\Delta_f$ defined on $\Lambda$.

For $k \in \mathbb{N}$, define $V_k \coloneqq \Lambda[y_1, \dots, y_k]=\Lambda\otimes \mathbb{Q}[y_1,\dots,y_k]$, so that $V_0=\Lambda$. For $k\geq 2$ and $1 \leq i \leq k-1$, let \[T_i \coloneqq \Upsilon_{y_i y_{i+1}} \colon V_k \rightarrow V_k.\] 
Notice that the $T_i$ are invertible operators (see \cite[Section~4]{Carlsson-Mellit-ShuffleConj-2015} for an explicit formula of the inverse).

For $k\geq 0$, we define the operators $d_+ \colon V_k \rightarrow V_{k+1}$ as
\begin{align*}
	(d_+ F)[X] & \coloneqq T_1 T_2 \cdots T_k (F[X + (q-1) y_{k+1}])\qquad \text{ for any }F[X]\in V_k,
\end{align*}
while for $k\geq 1$ we define the operators $d_- \colon V_k \rightarrow V_{k-1}$ as
\begin{align*}
	(d_- F)[X] & \coloneqq -F[X - (q-1)y_k] \sum_{i\geq 0} \left.  (-1/y_k )^{i}e_i[X]  \right|_{{y_k}^{-1}}\qquad \text{ for  any }F[X]\in V_k.
\end{align*}
Finally, for $k\geq 1$ we define the operators $\varphi:V_k\to V_k$ as
\begin{equation} \label{eq:phi_def}
\varphi:=\frac{1}{q-1}(d_-d_+-d_+d_-).
\end{equation}

Notice that, following Carlsson and Mellit, in the notation of the operators $T_i$, $d_-$, $d_+$ and $\varphi$, the $k$ indicating the domain $V_k$ does not appear. To keep track of this in our arguments, when the domain of (compositions of) such operators is $V_k$ we say that they have \emph{degree} $k$.

We record here a few identities satisfied by these operators that we are going to need. A proof of these identities can be found in \cite[Lemma~5.3]{Carlsson-Mellit-ShuffleConj-2015} and \cite[Section~3.1]{Carlsson_Gorsky_Mellit}. 
\begin{proposition}
In all the following identities, all sides have degree $k$: we have
\begin{align}
\label{eq:Ti_rel}	(T_i-1)(T_i+q)& =0 \\
\label{eq:Ti_phi}	\varphi T_i & = T_{i+1} \varphi\quad  (i\leq k-2)\\
\label{eq:phi2}	\varphi^2 T_{k-1} & = T_1 \varphi^2\\
\label{eq:Ti_dm}	d_-T_i & =T_id_-\quad (1\leq i\leq k-2)\\
\label{eq:dm2}	d_-^2T_{k-1} & =d_-^2\\
\label{eq:dphiT}	d_-\varphi T_{k-1} & =q\varphi d_-\\
\label{eq:dplusphi}  T_1\varphi d_+ & = q d_+\varphi   \, .
\end{align}
\end{proposition}
From these we can deduce the following identities, which will be useful in the sequel.
\begin{proposition}
In degree $k\geq 2$ we have
\begin{align}
\label{eq:auxdplus}	\varphi d_+ & = T_1d_+\varphi +(q-1) d_+\varphi \\
\label{eq:auxdminus} \varphi d_-T_{k-1} & =d_-\varphi - (q-1)\varphi d_- \, .
\end{align}
\end{proposition}
\begin{proof}
Observe that for any $i$
\begin{align*}
	(T_{i}-1)(T_{i}+q) = T_{i}^2+(q-1)T_{i}-q
\end{align*}
so that, multiplying by $T_{i}^{-1}$,
we can rewrite \eqref{eq:Ti_rel} as
\begin{equation} \label{eq:auxiliar}
T_{i}=qT_{i}^{-1} -(q-1). \end{equation}
Now multiplying \eqref{eq:dplusphi} by $T_1^{-1}$ we get
\begin{align*}
 \varphi d_+ & = q 	T_1^{-1} d_+\varphi\\
 \text{(using \eqref{eq:auxiliar})} & = T_1  d_+\varphi + (q-1) d_+\varphi
\end{align*}
giving \eqref{eq:auxdplus}.

On the other hand, using again \eqref{eq:auxiliar} we get
\begin{align*}
	\varphi d_-T_{k-1} & = q\varphi d_- T_{k-1}^{-1} -(q-1) \varphi d_-\\
\text{(using \eqref{eq:dphiT})}	& = d_-\varphi T_{k-1} T_{k-1}^{-1} -(q-1) \varphi d_- \\
& = d_-\varphi  -(q-1) \varphi d_-
\end{align*}
giving \eqref{eq:auxdminus}.
\end{proof}

Moreover, by \cite[Lemma~5.4]{Carlsson-Mellit-ShuffleConj-2015}, for $F\in V_k$
\begin{equation} \label{eq:phi_yk}
\varphi F=T_1T_2\cdots T_{k-1}(-y_kF).
\end{equation}

\section{Path operators}

Consider the family $\mathcal{P}$ of \emph{paths} defined in Section~\ref{sec:LLT}, i.e.\ Schr\"{o}der paths with no diagonal steps on the main diagonal $x=y$.

\medskip

\emph{As we will only use paths coming from $\mathcal{P}$, in the rest of this article the word ``path'' without further specifications will be used to mean ``path in $\mathcal{P}$''.}

\medskip

It will be convenient to encode such paths with \emph{words} in the alphabet $\{-,0,+\}$ where a north step corresponds to a `$-$', a northeast step to a `$0$', and an east step to a `$+$'. For example the path in Figure~\ref{fig:LLT2} gets encoded by $(-,-,0,0,-,+,-,-,+,+,0,0,-,+,+,+)$.

The \emph{degree} of a step of a path $P\in \mathcal{P}$ is the ordinate of the highest rightmost point of the step. Equivalently, it is the number of `$-$' minus the number of `$+$' occurring in the word encoding $P$ weakly to the left of the considered step. Observe that, by definition of $\mathcal{P}$, every `$0$' has positive degree, while the highest rightmost step of a path has always degree $0$.

A \emph{partial path} is simply a prefix of a path $P\in \mathcal{P}$, i.e.\ it consists of the first few consecutive steps of a path in $\mathcal{P}$ (so that its highest rightmost step can possibly have positive degree). The \emph{degree} of a partial path is defined to be the degree of its highest rightmost step.

A path is said to be \emph{non-touching} if it touches the base line only at the beginning and at the end, equivalently if its only step of degree $0$ is its rightmost one. 

Observe that any path $P$ can be decomposed in a unique way as the concatenation $P=P_1P_2\cdots P_k$ of non-touching paths $P_i$. 
\begin{definition}
We associate to each path $P\in \mathcal{P}$ a \emph{path operator} $d_P$ on $V_0=\Lambda$, by replacing each `$-$' with a $d_-$, each `$0$' with a $\varphi$, and each `$+$' with a $d_+$. 

We define similarly \emph{partial path operators} $d_P$ associated to partial paths $P$, where now $d_P:V_k\to V_0$, where $k$ is the degree of $P$.
\end{definition}
For example, if $P$ is encoded by $(-,-,0,0,+,0,+)$, then we get $d_P=d_-d_-\varphi \varphi d_+ \varphi d_+$.

\medskip

The main result that we are going to need about path operators is the following theorem of Carlsson and Mellit which is implict in \cite{Carlsson-Mellit-ShuffleConj-2015}, and made explicit in \cite{Garsia_Haglund_Qiu_Romero}.
\begin{theorem} \label{thm:CM_LLT}
Let $\mathbf{\nu}$ be a tuple of vertical strips. Then
\begin{equation}
G_{\mathbf{\nu}}[X;q]=d_{P_{\mathbf{\nu}}}(1),
\end{equation}
where $P_{\mathbf{\nu}}$ is defined in Section~\ref{Sec:Schroeder_LLT}.
\end{theorem}
Thanks to this theorem, we can work with path operators in order to prove the main result of this article. 

\section{$e$-positivity}

In this section we prove the main result of this article, i.e.\ the $e$-positivity of $G_{\mathbf{\nu}}[X;q+1]$ where $G_{\mathbf{\nu}}[X;q]$ is a vertical strip LLT polynomial.

\medskip

We start with a lemma.
\begin{lemma} \label{lem:ekoperator}
	For $m\geq 0$, the path operator $d_-\varphi^m d_+$  equals the operator of multiplication by $e_{m+1}$.
\end{lemma}
\begin{proof}
Using \eqref{eq:phi_yk} in degree $1$, we get the equality in degree $0$
\[d_-\varphi^m d_+=(-1)^m d_-y_1^m d_+.\]
	
For any $F[X]\in \Lambda=V_0$ we compute
\begin{align*}
	d_-\varphi^m d_+F[X] & = (-1)^md_-y_1^m d_+F[X] \\
	& = (-1)^md_-y_1^m F[X + (q-1) y_{1}]\\
	& = (-1)^{m+1} y_1^m F[X + (q-1) y_{1}- (q-1)y_1] \sum_{i\geq 0} \left.  (-1/y_1 )^{i}e_i[X]  \right|_{{y_1}^{-1}} \\
	& = F[X] \sum_{i\geq 0} \left.  (-1)^{m+i+1} (y_1)^{m-i}e_i[X]  \right|_{{y_1}^{-1}}\\
	& =e_{m+1}F[X].
\end{align*}
\end{proof}

The key ingredient of our main result is the following lemma.

\begin{lemma} \label{lem:crux}
	Consider a partial path operator $d_P$ of the form
	\[d_P=d_-\varphi^{a_1} d_-\varphi^{a_2} d_-\cdots  d_-\varphi^{a_{s-1}} d_-\varphi^{a_s} d_+\]
	with $s\geq 2$, and $a_i\geq 0$ for all $i=1,2,\dots,s$. 
	
	If $a_s=0$ then
	\begin{equation} \label{eq:case0}
	d_P=d_-\varphi^{a_1} d_-\cdots  d_-\varphi^{a_{s-1}}  d_+ d_- +(q-1)d_-\varphi^{a_1} d_-\cdots  d_-\varphi^{a_{s-1}}  \varphi.
	\end{equation}
	
	If $a_s>0$ then either
	\begin{equation} \label{eq:case1}
	d_P=qd_-\varphi^{a_1} d_-\cdots  d_-\varphi^{a_{s-1}} d_-\varphi^{a_{s}-1} d_+ \varphi  ,
	\end{equation}
	or
	\begin{align}\label{eq:case2} d_P & =q^\epsilon d_-\varphi^{a_1} d_-\cdots  d_-\varphi^{a_{s-1}-1} d_-\varphi^{a_{s}} d_+ \varphi  \\
	\notag	& + (q^\epsilon-1)d_-\varphi^{a_1} d_-\cdots  d_-\varphi^{a_{s-1}} d_-\varphi^{a_{s}-1} d_+ \varphi ,
	\end{align}
	with $a_{s-1}>0$ and $\epsilon\in \{0,1\}$, or
	\begin{align} \label{eq:lem_case3}
		d_P & =q^\epsilon d_-\varphi^{a_1} d_-\cdots d_-\varphi^{a_i-1} d_- \varphi^{a_{i+1}+1} d_-\cdots  d_-\varphi^{a_{s-1}} d_-\varphi^{a_{s}-1} d_+ \varphi\\
	\notag	&  + (q^\epsilon-1)d_-\varphi^{a_1} d_-\cdots  d_-\varphi^{a_{s-1}} d_-\varphi^{a_{s}-1} d_+ \varphi
	\end{align}
	for some $1\leq i\leq s-2$ for which $a_i>0$, and $\epsilon\in \{0,1\}$.
\end{lemma}

\begin{proof}
	If $a_s=0$, then we can use the definition \eqref{eq:phi_def} of $\varphi$ to get
	\begin{align*}
		& \hspace{-0.5cm} d_-\varphi^{a_1} d_-\cdots  d_-\varphi^{a_{s-1}} d_-\varphi^{a_s} d_+ =\\
		& = d_-\varphi^{a_1} d_-\cdots  d_-\varphi^{a_{s-1}} d_- d_+ \\
		& = d_-\varphi^{a_1} d_-\cdots  d_-\varphi^{a_{s-1}}  (d_- d_+ - d_+d_-)  +d_-\varphi^{a_1} d_-\cdots  d_-\varphi^{a_{s-1}}  d_+ d_- \\
	    & = (q-1)d_-\varphi^{a_1} d_-\cdots  d_-\varphi^{a_{s-1}}  \varphi +d_-\varphi^{a_1} d_-\cdots  d_-\varphi^{a_{s-1}}  d_+ d_-,
	\end{align*}
	which gives \eqref{eq:case0}.
	
	If $a_s>0$, then we can use \eqref{eq:auxdplus} to get
	\begin{align*}
		& \hspace{-0.5cm} d_-\varphi^{a_1} d_-\cdots  d_-\varphi^{a_{s-1}} d_-\varphi^{a_s} d_+ =\\
		& = d_-\varphi^{a_1} d_-\cdots  d_-\varphi^{a_{s-1}} d_-\varphi^{a_{s}-1} \textcolor{red}{\varphi d_+} \\
		& = d_-\varphi^{a_1} d_-\cdots  d_-\varphi^{a_{s-1}} d_-\varphi^{a_{s}-1} \textcolor{blue}{T_1d_+ \varphi }  + \textcolor{blue}{(q-1)}d_-\varphi^{a_1} d_-\cdots  d_-\varphi^{a_{s-1}} d_-\varphi^{a_{s}-1} \textcolor{blue}{d_+ \varphi} .
	\end{align*}
	Now, in order to get rid of $T_1$ in the first summand, we can use the commutation relations \eqref{eq:Ti_phi}, \eqref{eq:phi2} and \eqref{eq:Ti_dm}, i.e.\
	\[ \varphi T_i = T_{i+1} \varphi\quad  (i\leq k-2),\qquad \varphi^2 T_{k-1} = T_1 \varphi^2, \qquad d_-T_i=T_id_-\quad (1\leq i\leq k-2)  \]
	to commute the $T_1$ to the left, possibly changing it into another $T_i$, until we reach one of the following situations: 
	\begin{enumerate}
		\item in degree $k\geq 2$, we reach $\cdots d_-^2T_{k-1}$: in this case the relation \eqref{eq:dm2}, i.e.\ $d_-^2T_{k-1}=d_-^2$, allows us to simply remove the $T_{k-1}$;
		
		\item in degree $k\geq 2$, we reach $\cdots d_-\varphi T_{k-1}$: in this case the relation \eqref{eq:dphiT}, i.e.\ $d_-\varphi T_{k-1}=q\varphi d_-$, allows us again to remove the $T_{k-1}$, getting a factor $q$;
		
		\item in degree $k\geq 3$, we reach $\cdots \varphi d_- T_{k-1}$: in this case the relation \eqref{eq:auxdminus}, i.e.\ \[ \varphi d_- T_{k-1}= d_- \varphi -(q-1) \varphi d_-, \]
		will give a summand that cancels the one that we left out before, leaving us again with a summand without $T_{k-1}$.
	\end{enumerate}
	In situation (1) we get \eqref{eq:case1}, while in situations (2) and (3) we get the remaining claimed possibilities, i.e. \eqref{eq:case2} and \eqref{eq:lem_case3}, with $\epsilon=1$ and $\epsilon=0$ respectively.
\end{proof}

\begin{example} \label{ex:lem}
Consider the path $P\in \mathcal{P}$ encoded by $(-,0,-,-,0,0,0,0,+,+,+)$. We compute
\begin{align*}
	d_P & = d_-\varphi d_-^2\varphi^4 d_+^3\\
	& =	d_-\varphi d_-^2\varphi^3 \textcolor{red}{\varphi d_+}d_+^2\\
\text{(using \eqref{eq:auxdplus})}	& =	d_-\varphi d_-^2\varphi^3 \textcolor{blue}{T_1 d_+ \varphi}d_+^2 +\textcolor{blue}{(q-1)} d_-\varphi d_-^2\varphi^3  \textcolor{blue}{d_+ \varphi}d_+^2\\
	& =	d_- \varphi d_-^2\varphi^2\textcolor{red}{\varphi T_1} d_+ \varphi d_+^2 + (q-1) d_-\varphi d_-^2\varphi^3 d_+ \varphi d_+^2\\
\text{(using \eqref{eq:Ti_phi})}	& =	d_-\varphi d_-^2\varphi^2\textcolor{blue}{T_2 \varphi} d_+ \varphi d_+^2 + (q-1) d_-\varphi d_-^2\varphi^3 d_+ \varphi d_+^2 \\
& =	d_-\varphi d_-^2\textcolor{red}{\varphi^2 T_2 } \varphi d_+ \varphi d_+^2 + (q-1) d_-\varphi d_-^2\varphi^3 d_+ \varphi d_+^2 \\
\text{(using \eqref{eq:phi2})} & =	d_-\varphi d_-^2\textcolor{blue}{T_1\varphi^2 } \varphi d_+ \varphi d_+^2 + (q-1) d_-\varphi d_-^2\varphi^3 d_+ \varphi d_+^2 \\
& =	d_-\varphi d_- \textcolor{red}{d_-T_1}\varphi^3 d_+ \varphi d_+^2 + (q-1) d_-\varphi d_-^2\varphi^3 d_+ \varphi d_+^2\\
\text{(using \eqref{eq:Ti_dm})}& =	d_-\varphi d_- \textcolor{blue}{T_1 d_-}\varphi^3 d_+ \varphi d_+^2 + (q-1) d_-\varphi d_-^2\varphi^3 d_+ \varphi d_+^2\\
& =	d_-\textcolor{red}{\varphi d_- T_1} d_- \varphi^3 d_+ \varphi d_+^2 + (q-1) d_-\varphi d_-^2\varphi^3 d_+ \varphi d_+^2\\
\text{(using \eqref{eq:auxdminus})}& =	d_-\textcolor{blue}{d_- \varphi } d_- \varphi^3 d_+ \varphi d_+^2 \textcolor{blue}{-(q-1)}d_-\textcolor{blue}{\varphi d_- } d_- \varphi^3 d_+ \varphi d_+^2\\
& + (q-1) d_-\varphi d_-^2\varphi^3 d_+ \varphi d_+^2\\
& =	d_-^2 \varphi  d_- \varphi^3 d_+ \varphi d_+^2\, ,
\end{align*}
which gives the case \eqref{eq:lem_case3} with $\epsilon=0$.
\end{example}

We are now able to prove our main result.

\begin{theorem} \label{thm:main}
	If $d_P$ is a path operator, then the symmetric function $\left. d_P(1)\right|_{q=q+1}$ is $e$-positive.
\end{theorem}
\begin{proof}
Let $P\in \mathcal{P}$, and consider its decomposition $P=P_1P_2\cdots P_k$ in the concatenation of non-touching paths $P_i$. Clearly $d_P=d_{P_1}d_{P_2}\cdots d_{P_k}$. Consider the leftmost $d_+$ of degree $\geq 1$ in $d_P$, and suppose that it occurs in $d_{P_i}$ with $1\leq i\leq k$. By construction we have that all the $d_{P_j}$ with $j<i$ are of the form $d_-\varphi^{m_j}d_+$ for some $m_j\geq 0$. The prefix of $d_{P_i}$ that is cut by our $d_+$ (included) is a partial path operator of the form in the hypothesis of Lemma~\ref{lem:crux}. Therefore we can apply the lemma to make a substitution and replace our $d_P$ with a sum of one or two new path operators, each multiplied by a coefficient in $\{1,q,q-1\}$. The new operator paths that occur all have one of the following three properties:
\begin{enumerate}
	\item there is one $d_+$ less occurring, as it got replaced together with a $d_-$ by a $\varphi$;
	
	\item the $d_+$ occurs one place to the left of what it used to, in the same degree, as it passed by a $\varphi$;
	
	\item the $d_+$ occurs in a lower degree, as it passed by a $d_-$.
\end{enumerate}
Notice that nothing before $d_{P_i}$ and nothing after our $d_+$ got modified in our new path operators.

Therefore, iterating this process on the new paths operators, in finitely many steps either our $d_+$'s disappear as in (1), or they reach degree $0$, in which case they cut a prefix of the form $d_-\varphi^{m_i}d_+$ with $m_i\geq 0$.

Hence, until among our path operators there are $d_+$'s of degree $\geq 1$, we can iterate this all process.

At the end, in finitely many steps, we will have reached a combination of path operators in which every $d_+$ has degree $0$, i.e.\ these operators are compositions of operators of the form $d_-\varphi^{m}d_+$ with $m\geq 0$, with coefficients in $\mathbb{N}[q,q-1]=\mathbb{N}[q-1]$.

By Lemma~\ref{lem:ekoperator} the operators of the form $d_-\varphi^{m}d_+$ with $m\geq 0$ are multiplications by $e_{m+1}$, so in the end, applying the resulting formula for $d_P$ to the constant $1$ we ended up expanding the resulting symmetric function in the elementary symmetric function basis $\{e_\mu\}_\mu$ with coefficients in $\mathbb{N}[q-1]$. The substitution $q\mapsto q+1$ will therefore result in the $e$-positivity, as claimed.
\end{proof}
\begin{remark}
Observe that our proof gives actually an algorithm to expand $d_P(1)$ (and hence $\left.d_P(1)\right|_{q=q+1}$) in the elementary symmetric function basis.
\end{remark}
\begin{example}
Consider the path $P\in \mathcal{P}$ encoded by $(-,0,-,0,+,+)$. We start with the leftmost $d_+$, which appears in degree $1$: we compute
\begin{align*}
	d_P & = d_- \varphi  d_- \varphi d_+^2\\
	\text{(using \eqref{eq:auxdplus})} & = d_- \varphi  d_-  \textcolor{blue}{T_1d_+ \varphi} d_+ +\textcolor{blue}{(q-1)} d_- \varphi  d_-  \textcolor{blue}{d_+ \varphi} d_+ \\
	\text{(using \eqref{eq:auxdminus})} & = d_-   \textcolor{blue}{d_- \varphi} d_+ \varphi  d_+ \textcolor{blue}{-(q-1)} d_- \textcolor{blue}{\varphi  d_-}   d_+ \varphi  d_+ + (q-1)  d_- \varphi  d_-  d_+ \varphi  d_+ \\
	& = d_-^2 \varphi d_+ \varphi  d_+
\end{align*}
which is the case \eqref{eq:case2} with $\epsilon=0$. We can iterate on the leftmost $d_+$, which is still in degree $1$, to get
\begin{align*}
	d_P& = d_-^2 \varphi d_+ \varphi  d_+ \\
	\text{(using \eqref{eq:auxdplus})}& = d_-^2 \textcolor{blue}{T_1 d_+ \varphi} \varphi  d_+ + \textcolor{blue}{(q-1)}d_-^2 \textcolor{blue}{d_+ \varphi} \varphi  d_+\\
	\text{(using \eqref{eq:dm2})}& = d_-^2 d_+ \varphi^2  d_+ + (q-1) d_-^2 d_+ \varphi^2  d_+\\
	& =q d_-^2 d_+ \varphi^2  d_+ 
\end{align*}
which is the case \eqref{eq:case1}.

Iterating again on the leftmost $d_+$, which is in degree $1$, we get
\begin{align*}
	d_P&=q d_-^2 d_+ \varphi^2  d_+ \\
	\text{(using \eqref{eq:phi_def})}& =q d_- \textcolor{blue}{d_+ d_-} \varphi^2  d_+ + q\textcolor{blue}{(q-1)} d_-\textcolor{blue}{\varphi} \varphi^2  d_+\\
	& =q d_- d_+ d_-  \varphi^2  d_+ + q (q-1)  d_- \varphi^3  d_+,
\end{align*}
which is the case \eqref{eq:case0}.

Now all the $d_+$'s appear in degree $0$, so, applying Lemma~\ref{lem:ekoperator}, we finally get
\begin{align*}
	d_P(1)&  =q d_- d_+ d_-  \varphi^2  d_+(1) + q (q-1)  d_- \varphi^3  d_+(1) \\
	& =qe_1e_3+q(q-1)e_4.
\end{align*}
\end{example}

The following corollary is now immediate.
\begin{corollary}
	If $G_{\mathbf{\nu}}[X;q]$ is a vertical strip LLT polynomial, then $G_{\mathbf{\nu}}[X;q+1]$ is $e$-positive.
\end{corollary}
\begin{proof}
	The statement follows by combining Theorem~\ref{thm:CM_LLT} and the previous theorem.
\end{proof}

\section{Some consequences}

As we already mentioned in the introduction, our result implies several $e$-positivity conjectured in \cite{Alexandersson_Panova_cycles}, \cite{Bergeron_Open_Questions} and \cite{Garsia_Haglund_Qiu_Romero}. For example, we have the following corollary.  
\begin{corollary}
All the conjectured $e$-positivities left open in the introduction of \cite{Garsia_Haglund_Qiu_Romero} (the (7) conditioned to the validity of the Delta conjecture) are true.
\end{corollary}
\begin{proof}[Sketch of the proof]
All the occurring symmetric functions are either (special cases of) vertical strip LLT polynomials or positive combinations (i.e.\ with coefficients in $\mathbb{N}[q,t]$) of them: see for example \cite{Haglund-Book-2008} and \cite{BGSX_Rational_Shuffle} for more informations on these decompositions.
\end{proof}

\medskip

In a different direction, the following corollary is an immediate consequence of the proof of Theorem~\ref{thm:main}.
\begin{corollary}
	The path operator $d_P$ is simply a multiplication operator by $d_P(1)$.
\end{corollary}
\begin{proof}
In the proof of Theorem~\ref{thm:main} we showed that $d_P$ is a linear combination of compositions of path operators of the form $d_-\varphi^{m}d_+$ with $m\geq 0$, which are multiplication operators by Lemma~\ref{lem:ekoperator}. Therefore $d_P$ is itself a multiplication operator, obviously by $d_P(1)$.
\end{proof}
\begin{remark}
This fact is probably related to the same property conjectured by Bergeron in \cite{Bergeron_Open_Questions} for the specialization at $t=1$ of related symmetric functions. We do not discuss this further in this article.
\end{remark}
An interesting consequence of the last corollary, which is not at all clear from the definitions, is that the path operators actually commute with each others.

\section*{Acknowledgment}

We thank Adriano Garsia for attracting our attention to this topic in general, and to this specific problem in particular.

\bibliographystyle{alpha}
\bibliography{Bibliography}

\begin{bibdiv}
\begin{biblist}

\bib{Alexandersson_lollipop}{article}{
      author={{Alexandersson}, Per},
       title={{LLT polynomials, elementary symmetric functions and melting
  lollipops}},
        date={2019Mar},
     journal={arXiv e-prints},
       pages={arXiv:1903.03998},
      eprint={1903.03998},
}

\bib{Alexandersson_Panova_cycles}{article}{
      author={Alexandersson, Per},
      author={Panova, Greta},
       title={L{LT} polynomials, chromatic quasisymmetric functions and graphs
  with cycles},
        date={2018},
        ISSN={0012-365X},
     journal={Discrete Math.},
      volume={341},
      number={12},
       pages={3453\ndash 3482},
         url={https://doi.org/10.1016/j.disc.2018.09.001},
      review={\MR{3862644}},
}

\bib{Bergeron_Open_Questions}{article}{
      author={Bergeron, Fran\c{c}ois},
       title={Open questions for operators related to rectangular {C}atalan
  combinatorics},
        date={2017},
        ISSN={2156-3527},
     journal={J. Comb.},
      volume={8},
      number={4},
       pages={673\ndash 703},
         url={https://doi.org/10.4310/JOC.2017.v8.n4.a6},
      review={\MR{3682396}},
}

\bib{BGSX_Rational_Shuffle}{article}{
      author={Bergeron, Francois},
      author={Garsia, Adriano},
      author={Sergel~Leven, Emily},
      author={Xin, Guoce},
       title={Compositional {$(km,kn)$}-shuffle conjectures},
        date={2016},
        ISSN={1073-7928},
     journal={Int. Math. Res. Not. IMRN},
      number={14},
       pages={4229\ndash 4270},
         url={https://doi.org/10.1093/imrn/rnv272},
      review={\MR{3556418}},
}

\bib{Carlsson_Gorsky_Mellit}{article}{
      author={{Carlsson}, Erik},
      author={{Gorsky}, Eugene},
      author={{Mellit}, Anton},
       title={{The $\mathbb{A}_{q,t}$ algebra and parabolic flag Hilbert
  schemes}},
        date={2017Oct},
     journal={arXiv e-prints},
       pages={arXiv:1710.01407},
      eprint={1710.01407},
}

\bib{Carlsson-Mellit-ShuffleConj-2015}{article}{
      author={Carlsson, Erik},
      author={Mellit, Anton},
       title={A proof of the shuffle conjecture},
        date={2018},
        ISSN={0894-0347},
     journal={J. Amer. Math. Soc.},
      volume={31},
      number={3},
       pages={661\ndash 697},
         url={https://doi.org/10.1090/jams/893},
      review={\MR{3787405}},
}

\bib{Garsia_Haglund_Qiu_Romero}{article}{
      author={{Garsia}, Adriano~M.},
      author={{Haglund}, James},
      author={{Qiu}, Dun},
      author={{Romero}, Marino},
       title={{$e$-Positivity Results and Conjectures}},
        date={2019Apr},
     journal={arXiv e-prints},
       pages={arXiv:1904.07912},
      eprint={1904.07912},
}

\bib{Haglund-Book-2008}{book}{
      author={Haglund, James},
       title={The {$q$},{$t$}-{C}atalan numbers and the space of diagonal
  harmonics},
      series={University Lecture Series},
   publisher={American Mathematical Society, Providence, RI},
        date={2008},
      volume={41},
        ISBN={978-0-8218-4411-3; 0-8218-4411-3},
        note={With an appendix on the combinatorics of Macdonald polynomials},
      review={\MR{2371044}},
}

\bib{HHL_JAMS}{article}{
      author={Haglund, J.},
      author={Haiman, M.},
      author={Loehr, N.},
       title={A combinatorial formula for {M}acdonald polynomials},
        date={2005},
        ISSN={0894-0347},
     journal={J. Amer. Math. Soc.},
      volume={18},
      number={3},
       pages={735\ndash 761},
         url={https://doi.org/10.1090/S0894-0347-05-00485-6},
      review={\MR{2138143}},
}

\bib{LLT_Original}{article}{
      author={Lascoux, Alain},
      author={Leclerc, Bernard},
      author={Thibon, Jean-Yves},
       title={Ribbon tableaux, {H}all-{L}ittlewood functions, quantum affine
  algebras, and unipotent varieties},
        date={1997},
        ISSN={0022-2488},
     journal={J. Math. Phys.},
      volume={38},
      number={2},
       pages={1041\ndash 1068},
         url={https://doi.org/10.1063/1.531807},
      review={\MR{1434225}},
}

\bib{Stanley-Book-1999}{book}{
      author={Stanley, Richard~P.},
       title={Enumerative combinatorics. {V}ol. 2},
      series={Cambridge Studies in Advanced Mathematics},
   publisher={Cambridge University Press, Cambridge},
        date={1999},
      volume={62},
        ISBN={0-521-56069-1; 0-521-78987-7},
         url={https://doi.org/10.1017/CBO9780511609589},
        note={With a foreword by Gian-Carlo Rota and appendix 1 by Sergey
  Fomin},
      review={\MR{1676282}},
}

\bib{Shareshian_Wachs_Original}{incollection}{
      author={Shareshian, John},
      author={Wachs, Michelle~L.},
       title={Chromatic quasisymmetric functions and {H}essenberg varieties},
        date={2012},
   booktitle={Configuration spaces},
      series={CRM Series},
      volume={14},
   publisher={Ed. Norm., Pisa},
       pages={433\ndash 460},
         url={https://doi.org/10.1007/978-88-7642-431-1_20},
      review={\MR{3203651}},
}

\end{biblist}
\end{bibdiv}

\end{document}